\documentclass[a4paper, reqno, 12pt]{amsart}

\usepackage{amsfonts, amsmath, amsthm}
\usepackage{latexsym}

\theoremstyle{plain}
\newtheorem{thm}{Theorem}[section]
\newtheorem{prp}[thm]{Proposition} 
\newtheorem{lem}[thm]{Lemma} 
\newtheorem{cor}[thm]{Corollary} 

\theoremstyle{definition}

\theoremstyle{remark}
\newtheorem{rmk}{Remark}[section]

\numberwithin{equation}{section}

\newcommand{\N}{\mathbb{N}}

\newcommand{\R}{\mathbb{R}}
\newcommand{\C}{\mathbb{C}}
\newcommand{\Sph}{\mathbb{S}}

\newcommand{\pa}{\partial}
\newcommand{\eps}{\varepsilon}
\newcommand{\jb}[1]{\langle #1 \rangle}
\newcommand{\dal}{\Box}
\newcommand{\ZP}{\N_0}

\DeclareMathOperator{\realpart}{\rm Re}
\DeclareMathOperator{\imagpart}{\rm Im}
\DeclareMathOperator{\supp}{\rm supp}

\title[The energy decay and asymptotics for wave equations]{The energy decay and 
 asymptotics for a class of semilinear wave equations
in two space dimensions}  

\author[S.~Katayama]{Soichiro Katayama}
\address{Department of Mathematics, Wakayama University,
             930 Sakaedani, Wakayama 640-8510, Japan}
\email{katayama@center.wakayama-u.ac.jp}
\author[D.~Murotani]{Daisuke Murotani}
\address{Department of Mathematics, Graduate School of Science, 
             Osaka University, 
             Toyonaka, Osaka 560-0043, Japan.} 

\email{d-murotani@cr.math.sci.osaka-u.ac.jp}
\author[H.~Sunagawa]{Hideaki Sunagawa}
\address{Department of Mathematics, Graduate School of Science, 
             Osaka University, 
             Toyonaka, Osaka 560-0043, Japan.} 
\email{sunagawa@math.sci.osaka-u.ac.jp}
 


\begin{document}
\begin{abstract} 
We consider semilinear wave equations with small initial data 
in two space dimensions. 
For a class of wave equations with cubic nonlinearity, we show the
global existence of small amplitude solutions, and 
give an asymptotic description of the solution 
as $t \to \infty$ uniformly in $x \in \R^2$. 
In particular, our result implies the decay of the energy 
when the nonlinearity is dissipative.
\end{abstract}
\maketitle
\section{Introduction}
We consider the Cauchy problem for the following type of semilinear wave equations 
with small data: 
\begin{align}
& \dal  u= F(\pa u), & & (t,x) \in (0, \infty) \times \R^2,
\label{WaveEq}\\
& u(0,x)=\eps f(x), \ \pa_t u(0,x)=\eps g(x), && x\in \R^2,
 \label{InitData}
\end{align}
where $u$ is a real-valued unknown function
of $(t,x)\in [0, \infty)\times \R^2$, $\dal=\pa_t^2-\Delta=\pa_t^2-(\pa_1^2+\pa_2^2)$,
and $\pa u=(\pa_0 u, \pa_1 u, \pa_2 u)$
with the notation
$\pa_0=\pa_t=\pa/\pa t$, and $\pa_{j}=\pa/ \pa x_j$ ($j=1,2$).
We suppose that $f,g \in C_0^{\infty}(\R^2;\R)$.
$\eps$ is a small positive parameter. 
We assume that the nonlinear term $F=F(q)$ is a $C^\infty$ function of $q=(q_0, q_1, q_2)\in \R^3$.

The local existence of classical solutions to \eqref{WaveEq}--\eqref{InitData}
is well known, and we are interested in the sufficient condition
for the global existence of the solutions, and 
also in the asymptotic behavior of global solutions.
We say that the {\it small data global existence} (or SDGE) holds if
for any $f,g\in C^\infty_0(\R^2)$ there is a positive constant $\eps_0$
such that \eqref{WaveEq}--\eqref{InitData} possesses a global solution $u\in C^\infty([0,\infty)\times \R^2)$ for $0<\eps\le \eps_0$.
The case of cubic nonlinearity,
that is to say $F(q)=O(|q|^3)$ near $q=0$, is critical in two space dimensions,
because SDGE holds for some nonlinearity and fails for others.
For example, SDGE does not hold for $F=(\pa_t u)^3$,
however SDGE holds when $F$ is cubic and 
the {\it null condition for cubic terms}
(we refer to it as the {\it cubic null condition}) is satisfied (see Godin~\cite{God93}).
To be more specific, we say that $F$ 
satisfies the cubic null condition if
$F^{(c)}(-1, \omega_1, \omega_2)=0$ for any 
$\omega=(\omega_1, \omega_2)\in \Sph^1$, 
where $F^{(c)}$ denotes the cubic homogeneous part of $F$, that is,
$$
 F^{(c)}(q)=\lim_{\lambda \to +0} \lambda^{-3} F(\lambda q).
$$
The typical example satisfying the cubic null condition is
$$
 F(\pa u)=\sum_{a=0}^2 c_a(\pa_a u)\left(|\pa_t u|^2-|\nabla_x u|^2\right)
$$ 
with arbitrary constants $c_a$.
The null condition was first introduced for systems of 
nonlinear wave equations with quadratic nonlinearity in three space dimensions as a sufficient condition to ensure SDGE (see Klainerman~\cite{Kla86} and
Christodoulou~\cite{Chr86}; note that 
the case of quadratic nonlinearity is critical in three space dimensions).
The terms satisfying the null condition form an important class of 
nonlinearity.
We do not go into details, but 
the global existence under the null condition for the quasi-linear systems, 
even with quadratic nonlinearity, in two space dimensions
is also studied by many authors (see \cite{Ali01:01}, \cite {Ali01:02}, 
\cite{Hos95}, \cite{Kat93}, and \cite{Kat95} for example).
 
Another important class of nonlinearity is the {\it nonlinear dissipation}.
To explain the situation clearly, we consider the following equation
in general space dimensions:
\begin{align}
& \dal  u= -|\pa_t u|^{p-1} (\pa_t u), & & (t,x) \in (0, \infty) \times \R^n,
\label{DisWaveEq}\\
& u(0,x)=\varphi(x), \ \pa_t u(0,x)=\psi(x), && x\in \R^n,
 \label{DisInitData}
\end{align}
where $p>1$.
It is known that there is a global solution $u\in C\bigl([0,\infty); H^1(\R^n)\bigr)\cap C^1\bigl([0,\infty); L^2(\R^n)\bigr)$ for $(\varphi, \psi)\in X_p:=H^2(\R^n)\times(H^1(\R^n)\cap L^{2p}(\R^n))$ (see Lions--Strauss~\cite{LioStr65} for instance). Here no smallness of the data is required.
We define the energy norm by
$$
\|u(t)\|_E^2:=\frac{1}{2}\int_{\R^n} \left(|\pa_t u(t,x)|^2+|\nabla_x u(t,x)|^2\right)dx.
$$
Mochizuki-Motai~\cite{MocMot95} proved that 
if $n\ge 1$ and $1<p\le 1+2/n$, then the energy decays to zero, namely
$\lim_{t\to\infty} \|u(t)\|_E=0$ for initial data belonging to a dense subset
of $X_p$ (see also
Todorova-Yordanov~\cite{TodYor07}). 
On the other hand,  in the case of $n\ge 2$ and $p>1+2/(n-1)$, 
it is also proved in \cite{MocMot95} that the energy does not decay for a 
class of small initial data in $X_p$.
To sum up, the result in \cite{MocMot95} for $n=2$ can be read as follows: 
The energy decays for initial data in a dense subset of $X_p$ if $1<p\le 2$,
while the energy does not decay for small initial data if $p>3$. There is a gap in
the conditions for $p$, and it is quite interesting to investigate what happens
for \eqref{DisWaveEq} with $n=2$ and $p=3$ (or equivalently \eqref{WaveEq} 
with $F(\pa u)=-|\pa_t u|^2\pa_t u$).
 
The result above suggests that SDGE holds for the nonlinearity 
$F(\pa u)=-|\pa_t u|^2\pa_t u$, though it does not satisfy the cubic null 
condition. Hence it is natural to expect that there is a sufficient condition 
for SDGE which is weaker than the cubic null condition 
and includes also the nonlinear dissipative terms. 
Agemi conjectured that the condition 
\begin{equation}
\label{AHK}
F^{(c)}(-1, \omega_1,\omega_2)\ge 0,
\quad \omega=(\omega_1,\omega_2)\in \Sph^1
\end{equation}
implies SDGE. This conjecture was proved to be true
by Hoshiga~\cite{Hos08} and Kubo~\cite{Kub07} independently.
Moreover some asymptotic pointwise behavior of global solutions under 
\eqref{AHK} is obtained in \cite{Kub07} 
(see also Hayashi-Naumkin-Sunagawa~\cite{HayNauSun08} and 
Sunagawa~\cite{Sun06} for related 
results on nonlinear Schr\"odinger equations and nonlinear
Klein-Gordon equations, respectively).
Let us restrict our attention to the case of the nonlinear dissipative term 
$F(\pa u)=-|\pa_t u|^2\pa_t u$.
Then, it follows from \cite{Kub07} that
$$
|\pa u(t,x)|\le Cr^{-1/2} \left(\log \frac{r}{|r-t-1|+1}\right)^{-1/2}, 
\quad r\ge \frac{t}{2}+1
$$
with some positive constant $C$, where $r=|x|$. 
This estimate is improved in \cite{Mur11} as follows:  
$$
|\pa u(t,x)|\le Ct^{-1/2} \min\left\{\left(\log t \right)^{-1/2}, \eps
(1+|r-t|)^{-\vartheta}\right\}
$$
for $(t,x)\in [2,\infty)\times \R^2$ with $0<\vartheta<1/2$. 
However this is still insufficient in order to say something about the decay 
of the energy.

In this paper we will further improve the pointwise estimate of $\pa u$ under 
the condition \eqref{AHK}, and show that the energy decays to zero 
if we consider \eqref{WaveEq} with $F(\pa u)=-|\pa_t u|^2\pa_t u$ 
(or \eqref{DisWaveEq} with $n=2$ and $p=3$ in other words)
for small $C^\infty_0$-data 
(in fact, complex-valued solutions will be considered).
\section{The main result and its applications}
\subsection{Global existence and asymptotic pointwise behavior}
In what follows, we consider the initial value problem 
\eqref{WaveEq}--\eqref{InitData} for complex-valued data with the nonlinearity
\begin{equation}
\label{CExp}
F(\pa u)=
\sum_{a,b,c=0}^2 p_{abc} (\pa_a u)(\pa_b u)(\overline{\pa_c u})
\end{equation}
with complex constants $p_{abc}$ in order that we can catch up two kinds of 
interesting nonlinearities 
$F(\pa u)=-|\pa_t u|^2(\pa_t u)$ and $F(\pa u)=i |\pa_t u|^2 (\pa_t u)$ 
(see Subsections~\ref{App02} and \ref{App03} below). 
Here and hereafter, $\overline{z}$ denotes the complex 
conjugate of $z \in \C$ and the symbol $i$ always stands for $\sqrt{-1}$. 
We also use the notation $\jb{z}=\sqrt{1+|z|^2}$. 
The following theorem is our main result.
\begin{thm}\label{Global}
Assume that \eqref{CExp} is satisfied. We also assume
\begin{equation}
\label{ComplexAHK}
\realpart F(\hat{\omega}) \ge 0, \quad
\hat{\omega}=
(\omega_0,\omega)\in \{-1\}\times \Sph^1.
\end{equation}
Let $\mu$ be a sufficiently small positive constant satisfying $0<\mu<1/10$, 
say. Then, for any $f,g\in C^\infty_0(\R^2;\C)$, there exists a positive 
constant $\eps_0$ such that \eqref{WaveEq}--\eqref{InitData} admits a unique global classical 
solution $u\in C^\infty\bigl([0,\infty)\times\R^2; \C\bigr)$ for any 
$\eps \in (0,\eps_0]$. Moreover, 
there is a function $P_0$ of $(\sigma, \omega)\in \R\times \Sph^1$, 
which satisfies 
\begin{align}
\label{Concl02}
|P_0(\sigma, \omega)|\le C\eps \jb{\sigma}^{\mu-1},\quad (\sigma, \omega)\in \R\times \Sph^1
\end{align}
with a positive constant $C$, such that
\begin{align}
\label{Concl01}
\pa u(t,x)=\hat{\omega}(x) t^{-1/2} P(\log t, |x|-t, |x|^{-1}x)
+O\bigl(\eps t^{4\mu-(3/2)}\jb{t-|x|}^{-3\mu}\bigr)
\end{align}
for any $(t,x)\in [1,\infty)\times \R^2$,
where 
$\hat{\omega}(x)=(-1, |x|^{-1}x_1, |x|^{-1} x_2)$ for $x=(x_1, x_2)\in \R^2$, and 
$P=P(\tau,\sigma, \omega)$ is defined as a solution to
\begin{align}
\begin{cases}
 \pa_{\tau} P(\tau,\sigma, \omega) =-\left({F(\hat{\omega})}/2\right)|P(\tau,\sigma,\omega)|^2 P(\tau,\sigma, \omega), & \tau>0, \\
 P(0,\sigma,\omega)=P_0(\sigma,\omega)& 
\end{cases}
\label{ode}
\end{align}
with $\hat\omega=(-1,\omega)$.
Here the constant $C$ is independent of $\eps$.
\end{thm}
%
\begin{rmk}
(1) It is well known that if $f=g=0$ for $|x|\ge R$, then $u(t,x)=0$ for $|x|\ge t+R$
(see H\"ormander~\cite{Hoe97} for instance).
Hence $t$ is equivalent to $\jb{t+|x|}$ in $\supp u(t,\cdot)$ when $t\ge 1$. \\
(2)
$P(\tau,\sigma,\omega)$ in the above theorem can be explicitly solved as
\begin{equation}
\label{ExpS}
 P(\tau,\sigma,\omega)=
 P_0(\sigma,\omega) 
 \frac{
 \exp\left(-i\Theta(\tau,\sigma, \omega) \right)
 }
 {\sqrt{1+\left(\realpart{F}(\hat{\omega})\right) |P_0(\sigma,\omega)|^2\tau}}
\end{equation}
with
\begin{equation}
\label{Phase}
\Theta(\tau,\sigma, \omega)=\frac{1}{2}\left(\imagpart F(\hat{\omega})\right)
  \int_{0}^{\tau} 
 \frac{|P_0(\sigma, \omega)|^2}{{1+\left(\realpart F(\hat{\omega})\right)|P_0(\sigma,\omega)|^2\tau'}}d\tau'.
\end{equation}
From \eqref{ExpS}, we have
\begin{equation}
\label{ExpS1}
|P(\tau,\sigma, \omega)|\le \frac{|P_0(\sigma,\omega)|}{\sqrt{1+\left(\realpart F(\hat{\omega})\right) |P_0(\sigma,\omega)|^2\tau}}.
\end{equation}
(3) We can add higher order nonlinear terms to \eqref{CExp}, 
but the result becomes slightly weaker from the viewpoint of 
the estimate for the remainder term: Theorem~\ref{Global} remains valid if 
we replace \eqref{CExp} by
$$
F(\pa u)
=\sum_{a,b,c=0}^2 p_{abc}(\pa_a u)(\pa_b u)\left(\overline{\pa_c u}\right)
+O(|\pa u|^4)\quad \text{ near $\pa u=0$},
$$
$F$ in \eqref{ComplexAHK} as well as in \eqref{ode} by $F^{(c)}$, 
and $O(\eps t^{4\mu-(3/2)} \jb{t-|x|}^{-3\mu})$ in \eqref{Concl01} by
$O(\eps t^{\mu-1}\jb{t-|x|}^{-1/2})$.
\qed
\end{rmk}
We will prove Theorem~\ref{Global} in Section~\ref{Proof0101}
after some preparation given in Sections~\ref{Preliminary}, \ref{Reduction}
and \ref{LemmaODE}.
Compared to the method by Kubo~\cite{Kub07}, the most different point in our proof is
the choice of the equation \eqref{ode} for the asymptotics. Careful treatment
of the factor $\jb{t-|x|}$ is also quite important in our improvement.

In the following subsections, we discuss what we can see from 
Theorem~\ref{Global} particularly in the cases of 
$F(\pa u)= -|\pa_t u|^2\pa_t u$ (Subsection~\ref{App02}) and $F(\pa u)= i|\pa_t u|^2\pa_t u$ 
(Subsection~\ref{App03}).
\subsection{The case of nonlinear dissipation: The decay of the energy}
\label{App02}
We focus on the case where the inequality in \eqref{ComplexAHK} is strict, 
i.e.,
\begin{equation}
\label{DampingAss}
C_0:=\min_{\omega\in \Sph^1} \realpart F(\hat{\omega})>0.
\end{equation}
The typical example satisfying \eqref{DampingAss} is 
$F(\pa u)=-|\pa_t u|^2(\pa_t u)$.
In this case, it follows from \eqref{Concl02} and \eqref{ExpS1} that
$$
|P(\tau,\sigma, \omega)|\le \min\left\{\frac{1}
{\sqrt{C_0 \tau}}, C\eps \jb{\sigma}^{\mu-1}\right\}.
$$
Hence, by \eqref{Concl01} we can find a positive constant $C$ 
such that
\begin{equation}
\label{DecayDamping}
|\pa u(t,x)|\le C t^{-1/2}\min\left\{(\log t)^{-1/2}, \eps \jb{t-|x|}^{\mu-1}\right\}
\end{equation}
for $(t,x)\in [2,\infty)\times \R^2$.
This estimate says
that $\pa u$ decays like $(t \log t)^{-1/2}$ along the line 
$l_\sigma:=\{(t,x);\, |x|-t=\sigma\}$ for each $\sigma\in \R$. 
On the other hand, for the solution $u_0$ to the free wave 
equation $\dal u_0=0$ with $C^\infty_0$-data, it is known that 
$\pa u_0$ decays at the rate of $t^{-1/2}$ along  $l_\sigma$. 
Hence \eqref{DecayDamping} tells us that $\pa u$ has a gain of 
$(\log t)^{-1/2}$ in the pointwise decay compared to $\pa u_0$. 
Moreover, from \eqref{DecayDamping} we obtain the following decay of the 
energy:
\begin{cor}\label{EnergyDecay}
Suppose that \eqref{CExp} and \eqref{DampingAss} are fulfilled,
and let $0<\mu<1/10$.
Then for any $f, g\in C^\infty_0(\R^2; \C)$, the global solution $u$ to
\eqref{WaveEq}--\eqref{InitData} satisfies
$$
\|u(t)\|_E^2\le C\eps^{\frac{1}{1-\mu}}(\log t)^{-\frac{1-2\mu}{2-2\mu}},\quad t\ge 2
$$
with some positive constant $C$, provided that $\eps$ is sufficiently small.
\end{cor}
The proof for this result will be given in Section~\ref{Omake}.
This corollary says that the energy non-decay result of 
Mochizuki-Motai~\cite{MocMot95} fails for $n=2$ and $p=3$.
It also suggests that the energy decay result in \cite{MocMot95}
holds also for $n=2$ and $2< p\le 3$, but this is still an open problem 
(note that the energy decay for $n=2$ and $p=3$ established 
here is only for small $C^\infty_0$-data).

\subsection{The case without dissipation: Logarithmic correction of the phase.}
\label{App03}
We say that the global solution $u$ to \eqref{WaveEq}--\eqref{InitData} 
is asymptotically free (in the energy sense) if there is $(\varphi_+, \psi_+)
\in \dot{H}^1(\R^2)\times L^2(\R^2)$ such that
$$
\lim_{t\to \infty} \|u(t)-u_+(t)\|_E=0,
$$
where $u_+$ is the solution to the free wave equation
$\dal u_+=0$ in the energy class with initial data 
$(u_+,\pa_t u_+)=(\varphi_+,\psi_+)$ at $t=0$. Here $\dot{H}^1(\R^2)$
is the completion of $C^\infty_0(\R^2)$ with respect to the norm
$\|\phi\|_{\dot{H}^1(\R^2)}=\|\nabla\phi\|_{L^2(\R^2)}$.
It is proved in Katayama~\cite{Kat11b} (see also \cite{Kat11a}) 
that $u$ is asymptotically free 
if and only if there is $U=U(\sigma,\omega)\in L^2(\R\times \Sph^1)$
such that
$$
\lim_{t\to\infty} \bigl\|\pa u(t,\cdot)-\hat{\omega}(\cdot)\widetilde{U}(t,\cdot)
\bigr\|_{L^2(\R^2)}=0,
$$
where $\hat{\omega}(x)=(-1, x_1/|x|, x_2/|x|)$ and 
$\widetilde{U}(t,x)=|x|^{-1/2} U(|x|-t, x/|x|)$.

Now we consider the case of
\begin{equation}
\realpart F(\hat{\omega})=0,\quad \omega\in \Sph^1,
\end{equation}
which is stronger than \eqref{ComplexAHK} but weaker than the cubic null 
condition. In this case, $P$ can be written as
\begin{equation}\label{ExpSS}
P(\tau, \sigma, \omega)=P_0(\sigma, \omega) \exp\left(
-i\widetilde\Theta(\sigma, \omega) \tau\right)
\end{equation}
with $\widetilde\Theta(\sigma, \omega)=
\left(\imagpart F(\hat\omega)\right)|P_0(\sigma, \omega)|^2/2(\in \R)$.
Hence we find from \eqref{Concl01} that $\pa u$ decays at the same rate of $t^{-1/2}$ as the 
derivatives of the free solution $u_0$ along the line $l_\sigma$ for each $\sigma$. 
By \eqref{Concl02}, \eqref{Concl01}, and \eqref{ExpSS}, we get
\begin{align}
\bigl\|
\pa u(t,\cdot)-\hat{\omega}(\cdot)\widetilde{P}(t, \cdot)
\bigr\|_{L^2(\R^2)}
\le & C\eps t^{2\mu-(1/2)}
\bigl\||\cdot|^{-1/2}\jb{t-|\cdot|}^{-\mu-(1/2)}\bigr\|_{L^2(\R^2)}
\nonumber\\
\le & C\eps t^{2\mu-(1/2)}\to 0\quad (t\to \infty),
\end{align}
where $\widetilde{P}(t, x)=|x|^{-1/2}P(\log t, |x|-t, x/|x|)$.
Moreover \eqref{Concl02} and \eqref{ExpSS} lead to
$$
\lim_{t\to\infty}\|P(\log t,\cdot,\cdot)\|_{L^2((-\infty, -t)\times \Sph^1)}=0.
$$
Therefore
we find that $u$ is asymptotically free
if and only if there is $U=U(\sigma,\omega)\in L^2(\R\times\Sph^1)$ such that
$$
\lim_{\tau\to\infty} \|P(\tau,\cdot,\cdot)-U(\cdot,\cdot)\|_{L^2(\R\times\Sph^1)}=0. 
$$

If we assume $\imagpart F(\hat{\omega})=0$ 
for any $\omega\in \Sph^1$ in addition,
then $P(\tau,\sigma,\omega)=P_0(\sigma, \omega)$.
Furthermore we see from \eqref{Concl02} that $P_0\in L^2(\R\times \Sph^1)$.
Hence we conclude that the solution $u$ is asymptotically free.
For this case, we have $F(\hat{\omega})\equiv0$ and the cubic null condition 
(for the complex case) is satisfied.
Typical examples are $F=(\pa_a u)\bigl(|\pa_t u|^2-|\pa_1u|^2-|\pa_2u|^2\bigr)$ or $(\overline{\pa_a u})\bigl((\pa_t u)^2-(\pa_1u)^2-(\pa_2u)^2\bigr)$,
as well as $F=(\pa_a u)\bigl((\pa_b u)(\overline{\pa_c u})-(\pa_c u)(\overline{\pa_b u})\bigr)$ for $a,b,c=0,1,2$.

The situation is different if $\imagpart F(\hat\omega)\not\equiv 0$.
For example, take $F(\pa u)=i|\pa_t u|^2(\pa_t u)$ so that 
$F(\hat{\omega})\equiv -i$. 
Then it is easy to show $\|u(t)\|_E=\|u(0)\|_E$. 
By \eqref{ExpSS} we get
$$
P(\tau,\sigma,\omega)=P_0(\sigma,\omega)\exp\left(\frac{i}{2}|P_0(\sigma,\omega)|^2\tau \right),
$$
and we can easily see that $P(\tau)$ does not converge to any function 
in $L^2(\R\times \Sph^1)$ as $\tau\to \infty$ unless $P_0\equiv 0$.
Because of the conservation of the energy, 
we can show that if $P_0\equiv 0$ then $(f,g)\equiv (0,0)$.
Hence we see that the global solution $u$ for small $\eps$ 
is not asymptotically free unless $(f, g)\equiv (0,0)$,
though the energy is preserved.
Such a phenomenon never occurs in the real-valued case.
\section{Preliminaries}\label{Preliminary}
We introduce
\begin{align*}
S:=& t\pa_t+\sum_{j=1}^2 x_j \pa_j
\\
L=&(L_1,L_2):=(t\pa_1+x_1\pa_t, t\pa_2+x_2\pa_t),
\ \Omega:= x_1\pa_2-x_2\pa_1,
\end{align*}
and we set
$$
 \Gamma=(\Gamma_0,\Gamma_1, \ldots, \Gamma_6)=(S, L_1, L_2, \Omega, \pa_0,\pa_1,\pa_2).
$$
With a multi-index $\alpha=(\alpha_0, \alpha_1, \ldots, \alpha_6)\in \ZP^7$, we write
$\Gamma^\alpha=\Gamma_0^{\alpha_0}\Gamma_1^{\alpha_1}\cdots \Gamma_6^{\alpha_6}$,
where $\ZP$ denotes the set of nonnegative integers.
For a smooth function $\psi=\psi(t,x)$ and a nonnegative integer $s$, we define
$$
|\psi(t,x)|_s=\sum_{|\alpha|\le s} |\Gamma^\alpha \psi(t,x)|,
\quad \|\psi(t)\|_s=\sum_{|\alpha|\le s} \|\Gamma^\alpha \psi(t,\cdot)\|_{L^2(\R^2)}.
$$
It is easy to see that $[\dal, \Omega]=[\dal, L_j]=[\dal, \pa_a]=0$
for $j=1,2$ and $a=0,1,2$, where $[A,B]=AB-BA$ for the operators $A$ and $B$.
We also have $[\dal, S]=2\dal$. Therefore for any $\alpha=(\alpha_0,\alpha_1,\ldots,\alpha_6)\in \ZP^7$ and a smooth function $\psi$, 
we have
\begin{equation}
\label{Comm01}
\dal \Gamma^\alpha \psi=(\Gamma_0+2)^{\alpha_0}\Gamma_1^{\alpha_1}\cdots\Gamma_6^{\alpha_6} \dal \psi=: \widetilde\Gamma^\alpha \dal \psi.
\end{equation}
We can check that we have $[\Gamma_a, \Gamma_b]=\sum_{0\le c\le 7} A^{ab}_c\Gamma_c$
and $[\Gamma_a, \pa_b]=\sum_{0\le c\le 2} B^{ab}_c \pa_c$ with appropriate constants
$A^{ab}_c$ and $B^{ab}_c$. Hence for any $\alpha, \beta\in \ZP^7$,
and any nonnegative integer $s$, there exist positive constants $C_{\alpha,\beta}$ 
and $C_s$ such that
\begin{align}
\label{Comm02}
& |\Gamma^\alpha\Gamma^\beta\psi(t,x)|\le C_{\alpha,\beta} |\psi(t,x)|_{|\alpha|+|\beta|},\\
\label{Comm03}
& C_s^{-1} |\pa \psi(t,x)|_s\le \sum_{0\le a\le 2}\sum_{|\gamma|\le s} |\pa_a\Gamma^\gamma\psi(t,x)|\le C_s|\pa \psi(t,x)|_s
\end{align}
for any smooth function $\psi$.

Using these vector fields, we obtain a good decay estimate for the
solution to the inhomogeneous wave equation. Let $0<T\le \infty$.
\begin{lem}[H\"ormander's $L^1$--$L^\infty$ estimate]\label{Ho}
Let $v$ be a smooth solution to
$$
\dal v(t,x)=\Psi(t,x), \quad (t,x)\in (0, T)\times \R^2
$$
with initial data $v=\pa_t v=0$ at $t=0$. 
Then there exists a universal positive constant $C$ such that
$$
\jb{t+|x|}^{1/2}|v(t,x)|\le C\int_0^t
\left(\int_{\R^2} \frac{|\Psi(\tau, y)|_1}{\jb{\tau+|y|}^{1/2}}dy\right) d\tau,
\quad 0\le t<T.
$$
\end{lem}
See H\"ormander \cite{Hoe88} for the proof (similar estimates for arbitrary 
space dimensions are also available there). 

The vector fields in $\Gamma$ also play an important role
in the reduction of the analysis of semilinear wave equations
to that of the corresponding ordinary differential equations (or ODEs in short).
We use the polar coordinates $(r,\theta)$
to write $x=(x_1,x_2)=(r\cos \theta, r\sin \theta)$
with $r=|x|$ and $\theta\in \R$. 
We put $\omega=(\omega_1,\omega_2)=x/|x|=(\cos\theta, \sin\theta)$. 
Then we have $\pa_r=\sum_{j=1}^2\omega_j\pa_j$, and $\pa_\theta=x_1\pa_2-x_2\pa_1=\Omega$.
We also have
\begin{equation}
\pa_1=\omega_1\pa_r-\frac{\omega_2}{r}\pa_\theta,\quad \pa_2=\omega_2\pa_r+\frac{\omega_1}{r}\pa_\theta.
\label{f001} 
\end{equation}
We put
$$
\pa_\pm=\pa_t\pm \pa_r.
$$
Then, for a smooth function $\psi$, we get
\begin{equation}
\label{pcd}
r^{1/2}\dal \psi
=\pa_+\pa_-(r^{1/2}\psi)-r^{-3/2}\left(\pa_\theta^2\psi+\frac{1}{4}\psi\right).
\end{equation}

Since
$$
\frac{1}{r}\pa_\theta=\frac{1}{r}(x_1\pa_2-x_2\pa_1)=\frac{1}{t}(\omega_1 L_2-\omega_2 L_1),
$$
it follows from \eqref{f001} that
\begin{equation}
\label{L202}
\sum_{j=1}^2|\pa_j\psi(t,x)-\omega_j\pa_r\psi(t,x)|\le \left|\frac{1}{r}\pa_\theta\psi(t,x)\right|
\le C\jb{t+r}^{-1}|\psi(t,x)|_1
\end{equation}
for any $\psi\in C^1([0, T)\times \R^2)$.

We put $\omega_0=-1$ and $\hat\omega=(\omega_0,\omega_1,\omega_2)$. 
For simplicity of exposition, we introduce
$$
D_\pm=\pm \frac{1}{2} \pa_{\pm}=\frac{1}{2} (\pa_r\pm \pa_t).
$$
\begin{lem}\label{Rewrite01}
There exists a positive constant $C$ such that
\begin{equation}
\label{f203}
 |\pa \psi(t,x)-\hat \omega D_-\psi(t,x)|
 \le C \jb{t+r}^{-1}|\psi(t,x)|_1 
\end{equation}
for $\psi\in C^1([0, T)\times\R^2)$ and 
$(t,x)\in [0, T)\times (\R^2\setminus\{0\})$. 
\end{lem}
%
\begin{proof}
Because $S=t\pa_t+r\pa_r$,
if we put $L_r=t\pa_r+r\pa_t=\sum_{j=1}^2 \omega_j L_j$, then we get
$$
D_+=\frac{1}{2}(\pa_r+\pa_t)=\frac{1}{2(t+r)} (S+L_r).
$$
Since we have $\pa_t=-D_-+D_+$ and $\pa_r=D_-+D_+$, we find
that $|\pa_t \psi(t,x)-\omega_0 D_-\psi(t,x)|$ and $|\pa_r \psi(t,x)-D_-\psi(t,x)|$
are bounded by $C \jb{t+r}^{-1} |\psi(t,x)|_1$.
Hence we obtain \eqref{f203} with the help of \eqref{L202}. 
\end{proof}

We define
$$
\Lambda_T:=\{(t,x)\in [0, T)\times \R^2; r\ge t/2\ge 1\}.
$$
Observe that we have $r^{-1}\le 4(1+t+r)^{-1}$ for $(t,x)\in \Lambda_T$.
We compute 
$D_-(r^{1/2}\psi)=r^{1/2}D_-\psi+r^{-1/2}\psi/4$. 
Thus from \eqref{f203} we obtain the following.
\begin{cor}\label{Rewrite02}
There is a positive constant $C$ such that
$$ 
 |r^{1/2}\pa \psi(t,x)-\hat\omega D_-\bigl(r^{1/2}\psi(t,x)\bigr)|
\le C \jb{t+r}^{-1/2}|\psi(t,x)|_1
$$
for $(t,x)\in \Lambda_T$ and $\psi\in C^1([0, T)\times \R^2)$.
\end{cor}

The following lemma is due to Lindblad \cite{Lin90}.
\begin{lem}\label{Rewrite03}
For any nonnegative integer $s$, there exists a positive constant $C_s$ such that
\begin{equation}
\label{L201}
|\pa \psi(t,x)|_s\le C_s \jb{t-r}^{-1}|\psi(t,x)|_{s+1},\quad (t,x)\in [0, T)\times \R^2
\end{equation}
for any $\psi\in C^{s+1}([0, T)\times \R^2)$.
\end{lem}
%
\begin{proof}
In view of \eqref{Comm02} and \eqref{Comm03}, it suffices
to prove \eqref{L201} for $s=0$. As in the proof of Lemma~\ref{Rewrite01},
we put $L_r=t\pa_r+r\pa_t$. Then we have
$$
(t-r)\pa_t\psi=\frac{1}{t+r}\left(tS-rL_r\right)\psi,\quad (t-r)\pa_r\psi=\frac{1}{t+r}
\left(tL_r-rS\right)\psi.
$$
Hence we get $\jb{t-r}(|\pa_t\psi|+|\pa_r\psi|)\le C|\psi|_1$, which leads to
\eqref{L201} for $s=0$ with the help of \eqref{L202}. 
\end{proof}
\section{Reduction to simplified equations}
\label{Reduction}
Let $0<T\le \infty$, and let  $u$ be the solution to \eqref{WaveEq}
on $[0, T)\times \R^2$.
We suppose that 
\begin{equation}
\label{Supp01}
 \supp f\cup \supp g \subset B_R 
\end{equation}
for some $R>0$, where $B_M=\{x\in \R^2;|x|\le M\}$ for $M>0$. Then, from the finite propagation property,
we have
\begin{equation}
\label{Supp02}
 \supp u(t,\cdot)\subset B_{t+R},\quad 0\le t<T.
\end{equation}

In what follows, we put $r=|x|$, $\omega=(\omega_1, \omega_2)=|x|^{-1}x$, and $\omega_0=-1$.
We write $\hat{\omega}=(\omega_0,\omega)=(-1,\omega)$.
We define 
$$
U(t, x):=D_-\bigl(r^{1/2} u(t, x) \bigr),\quad (t, x)\in [0,T)\times (\R^2\setminus\{0\}).
$$ 
From \eqref{WaveEq} and \eqref{pcd}, we get
\begin{equation}
\label{Red01}
\pa_+U(t, x)= -\frac{1}{2t}F(\hat{\omega}) |U(t, x)|^2U(t, x)+H(t, x),
\end{equation}
where $H=H(t,x)$ is given by
\begin{align*}
H=& -\frac{1}{2}
\left(r^{1/2}F(\pa u)-t^{-1}F(\hat{\omega})|U|^2U\right)
{}-\frac{1}{2r^{3/2}}\left(\pa_\theta^2 u+\frac{1}{4}u\right).
\end{align*}
\eqref{Red01} plays an important role in our analysis.

We also need the following to estimate the generalized derivative $\Gamma^\alpha u$
for a multi-index $\alpha=(\alpha_0, \alpha_1, \ldots, \alpha_6)\in \ZP^7$:
\eqref{Comm01} and \eqref{WaveEq} yield
$\dal(\Gamma^\alpha u)=\widetilde{\Gamma}^\alpha\left(F(\pa u)\right)$,
which leads to
\begin{equation}
\label{Red02}
\pa_+U^{(\alpha)}= -\frac{1}{2t}F(\hat{\omega}) 
\left( 2|U|^2U^{(\alpha)}+U^2\overline{U^{(\alpha)}}\right)
+H_\alpha,
\end{equation}
where we have put
$$
U^{(\alpha)}(t,x):= D_-\bigl(r^{1/2}\Gamma^\alpha u(t,x)\bigr),\\
$$
and $H_\alpha=H_\alpha(t,x)$ is given by
\begin{align*}
H_\alpha=& -\frac{1}{2}
\left(r^{1/2}\widetilde{\Gamma}^\alpha F(\pa u)-
t^{-1}F(\hat{\omega})\bigl(2|U|^2U^{(\alpha)}+U^2 \overline{U^{(\alpha)}}\bigr)\right)
\\
&
{}-\frac{1}{2r^{3/2}}\left(\pa_\theta^2 \Gamma^\alpha u+\frac{1}{4} \Gamma^\alpha u\right).
\end{align*}
Note that $U=U^{(\alpha)}$ if $|\alpha|=0$.

We put
\begin{equation}
\label{DefLamTR}
\Lambda_{T,R}:=\{(t,x)\in [0, T)\times \R^2;\, 1\le t/2\le r\le t+R\}.
\end{equation}
Note that we have
$$
(1+t+r)^{-1}\le r^{-1}\le 2t^{-1}\le 3(1+t)^{-1}\le 3(R+2)(1+t+r)^{-1}
$$
for $(t,x)\in \Lambda_{T,R}$. In other words, the weights $\jb{t+r}^{-1}$, 
$(1+t)^{-1}$, $r^{-1}$, and
$t^{-1}$ are equivalent to each other in $\Lambda_{T,R}$.
For nonnegative integer $s$, we define
\begin{equation}
\label{Basis300}
|u(t,x)|_{\sharp,s}:=|\pa u(t,x)|_s+\jb{t+r}^{-1}|u(t,x)|_{s+1}.
\end{equation}
For $|\alpha|\le s$, Corollary~\ref{Rewrite02} and \eqref{Comm03} immediately imply that
\begin{equation}
|U^{(\alpha)}(t,x)|\le C_s r^{1/2}|u(t,x)|_{\sharp,s}
\label{Basic301}
\end{equation}
for $(t,x)\in \Lambda_{T, R}$ with some positive constant $C_s$.
Our final goal in this section is to prove the following.
\begin{prp}\label{Pro301}
Suppose that \eqref{CExp} is fulfilled.
Let $s$ be a positive integer, and suppose that $1\le |\alpha|\le s$.
Then there are positive constants $C$ and $C_s$ such that
\begin{align}
\label{Basic321}
|H(t,x)|
 \le & C \left(t^{-1/2}|u|_{\sharp,0}^2|u|_1+t^{-3/2}|u|_2
\right),\\
|H_\alpha(t,x)|\le & C_s (t^{-1/2}|u|_{\sharp,s}^2|u|_{s+1}
+r^{1/2}|\pa u|_{s-1}^3+t^{-3/2}|u|_{s+2})
\label{Basic322}
\end{align}
for $(t,x)\in \Lambda_{T,R}$.
\end{prp}
Before we proceed to the proof of Proposition~\ref{Pro301}, we
show one lemma.
For $a,b,c=0,1,2$, we define
$$
Q_{abc}:=r^{1/2}(\pa_a u)(\pa_b u)(\overline{\pa_c u})-r^{-1}
\omega_a\omega_b\omega_c |U|^2U.
$$
For a multi-index $\alpha\in \ZP^7$, we also define
$$
\widetilde{Q}_{abc}^{(\alpha)}
:=r^{1/2}\Gamma^\alpha\left((\pa_a u)(\pa_b u)(\overline{\pa_c u})\right)
-r^{-1}\omega_a\omega_b\omega_c\left(2|U|^2U^{(\alpha)}+U^2\overline{U^{(\alpha)}}\right).
$$
\begin{lem}\label{L302}
Let $|\alpha|\le s$.
Then there are positive constants $C$ and $C_s$ such that
\begin{align}
\label{L3021}
|Q_{abc}|\le & C t^{-1/2}|u|_{\sharp,0}^2|u|_1,\\
\label{L3022}
|\widetilde{Q}_{abc}^{(\alpha)}|\le &
C_s (t^{-1/2} |u|_{\sharp, s}^2|u|_{s+1}
        {}+r^{1/2}|\pa u|_{s-1}^3)
\end{align}
for $(t,x)\in \Lambda_{T,R}$. 
\end{lem}
%
\begin{proof} 
Suppose that $(t,x)\in \Lambda_{T,R}$.
By \eqref{Basic301} and Corollary~\ref{Rewrite02}, we get
\begin{align}
|Q_{abc}|=& r^{-1}\bigl( (r^{1/2}\pa_a u)(r^{1/2}\pa_b u)(\overline{r^{1/2}\pa_c u})
{}-\omega_a\omega_b\omega_c|U|^2U\bigr)
\nonumber\\
\le & Cr^{-1}(|r^{1/2}\pa u|+|U|)^2|r^{1/2}\pa u-\hat \omega U|
\nonumber\\
\le & Ct^{-1/2}|u|_{\sharp,0}^2|u|_1.
 \label{Basic311}
\end{align}

Let $|\alpha|\le s$. In view of \eqref{Comm02} and \eqref{Comm03} we obtain
\begin{equation}
\label{Basic312}
\left|\Gamma^\alpha\bigl( (\pa_a u)(\pa_b u)(\overline{\pa_c u})\bigr)-R_{abc}^{(\alpha)}\right|
\le C_s|\pa u|_{[s/2]}^2|\pa u|_{s-1}\le C_s|\pa u|_{s-1}^3
\end{equation}
with some positive integer $C_s$,
where
$$
R_{abc}^{(\alpha)}=(\pa_a\Gamma^\alpha u)(\pa_b u)(\overline{\pa_c u})
{}+(\pa_a u)(\pa_b \Gamma^\alpha u)(\overline{\pa_c u})+(\pa_a u)(\pa_b u)(\overline{\pa_c \Gamma^\alpha u}).
$$
Similarly to \eqref{Basic311}, we get
\begin{align}
& |r^{1/2}R_{abc}^{(\alpha)}-r^{-1}\omega_a\omega_b\omega_c(2|U|^2U^{(\alpha)}+
U^2\overline{U^{(\alpha)}})|
\le C_st^{-1/2} |u|_{\sharp,s}^2|u|_{s+1}.
\label{Basic313}
\end{align}
The desired estimate \eqref{L3022} follows from \eqref{Basic312} and \eqref{Basic313}.
\end{proof}
%

\begin{proof}[Proof of Proposition~$\ref{Pro301}$]
Suppose that $(t,x)\in \Lambda_{T,R}$.
Recalling \eqref{CExp},
we get $r^{1/2}F(\pa u)-r^{-1}F(\hat{\omega})|U|^2U=\sum_{a,b,c} p_{abc} Q_{abc}$,
and Lemma~\ref{L302} yields
\begin{equation}
\label{PL301}
\left|r^{1/2}F(\pa u)-r^{-1}F(\hat{\omega})|U|^2U\right|\le Ct^{-1/2}|u|_{\sharp,0}^2|u|_1.
\end{equation}
Using \eqref{Basic301} and Lemma~\ref{Rewrite03} we get
\begin{align}
\left|(r^{-1}-t^{-1})F(\hat{\omega})|U|^2U\right|\le & Ct^{-1}r^{-1}\jb{t-r}r^{3/2}|u|_{\sharp,0}^3\nonumber\\
\le & Ct^{-1/2}|u|_{\sharp,0}^2 (\jb{t-r}|\pa u|+|u|_1) \nonumber\\
\le & Ct^{-1/2}|u|_{\sharp,0}^2 |u|_1 .
\label{PL302}
\end{align}
It is clear that we have
\begin{equation}
\label{PL304}
r^{-3/2}|\pa_\theta^2 u+u/4|\le Ct^{-3/2}|u|_2.
\end{equation}
Now \eqref{Basic321} follows from \eqref{PL301}, \eqref{PL302}, and \eqref{PL304}.

Using \eqref{L3022} instead of \eqref{L3021}, we can show \eqref{Basic322} in a similar manner. 
\end{proof}
%
\section{A key lemma for ordinary differential equations}
\label{LemmaODE}
Let $t_0\ge 1$. Motivated by \eqref{Red01}, we consider the initial value problem for 
the following kind of ODE:
\begin{align}
\label{ODE}
& z'(t)=-\frac{K}{2t} |z(t)|^2z(t)+{J}(t), \quad t_0<t<\infty,\\
\label{ODEinit}
& z(t_0)=z_0, 
\end{align}
where $K, z_0 \in \C$, and ${J} \in C\bigl([t_0, \infty);\C\bigr)$.
The following lemma is a refinement of the ones
obtained in Hayashi-Naumkin-Sunagawa~\cite{HayNauSun08} and
Sunagawa~\cite{Sun06}.
\begin{lem}\label{ODELemma02}
Let $\eps>0$, $\sigma\in \R$, $\rho>1$, $0<\mu<\rho-1$,
and $\kappa\ge 0$. 
Suppose that $\realpart K\ge 0$.
We assume that there are positive constants $K_0$ and $E_0$ such that
\begin{align*}
|K|\le & K_0,\\
|z_0|\le & E_0 \eps\jb{\sigma}^{-\kappa-\rho+1}, \\
|{J}(t)|\le & E_0 \eps \jb{\sigma}^{-\kappa} t^{-\rho},\quad t\ge t_0.
\end{align*}
We also assume that there is a positive constant $c_0$ such that
\begin{equation}
c_0^{-1}\jb{\sigma}<
t_0<c_0\jb{\sigma}.
\label{CondT0}
\end{equation}
Then \eqref{ODE}--\eqref{ODEinit}  
admits a unique global solution $z\in C^1([t_0,\infty);\C)$, and there is a
positive constant $C_0=C_0(E_0, c_0, \rho)$ such that we have
\begin{equation}
\label{ODEBound}
|z(t)|\le C_0\eps \jb{\sigma}^{-\kappa-\rho+1} ,\quad t_0\le t<\infty.
\end{equation}
Moreover 
there is a positive constant $\eps_1=\eps_1(K_0, E_0, c_0, \rho, \mu)$
such that if $0<\eps\le \eps_1$, then we can find $p_{0}\in \C$ 
satisfying
\begin{align}
\label{ODEaSymp}
|z(t)-p(\log t)|\le & C_1 \eps \jb{\sigma}^{-\kappa-\mu}t^{-\rho+\mu+1},\quad t\ge t_0, \\
\label{ODEprof}
|p_0|\le & C_1 \eps \jb{\sigma}^{-\kappa-\rho+1},
\end{align}
where $p(t)$ is a solution to
\begin{align}
\left\{\begin{array}{ll}
 p'(\tau)=-\frac{K}{2}|p(\tau)|^2p(\tau),& \tau>0,\\
 p(0)=p_{0},
\end{array}\right.
\label{ReducedODE}
\end{align}
and $C_1=C_1(K_0, E_0, c_0, \rho, \mu)$ is a positive constant.
\end{lem}
%
\begin{proof} 
We put $\nu=\kappa+\rho-1(>0)$.
Let $z$ be the solution to \eqref{ODE}--\eqref{ODEinit}
for $t_0<t<T_0$ with some $T_0(>t_0)$. 
Since $\realpart K\ge 0$, it follows from
\eqref{ODE} that
$$
\frac{d}{dt}(|z(t)|^2)=-\frac{\realpart K}{t}|z(t)|^4
+2\realpart({J}(t)\overline{z(t)})
\le 2|J(t)|\,|z(t)|,
$$
which yields
\begin{align*}
|z(t)|\le & |z_0|+
\int_{t_0}^t |{J}(\tau)|d\tau\le E_0\eps\left(\jb{\sigma}^{-\nu}
+ \jb{\sigma}^{-\kappa}\int_{t_0}^\infty \tau^{-\rho} d\tau\right)
\\
=& E_0\eps\left(\jb{\sigma}^{-\nu}+\jb{\sigma}^{-\kappa}\frac{t_0^{-\rho+1}}{\rho-1}\right)
\le E_0\left(1+\frac{c_0^{\rho-1}}{\rho-1}\right)\eps\jb{\sigma}^{-\nu}
\end{align*}
for $t_0\le t<T_0$. With this {\it a priori} bound and the local existence theorem for ODEs, we can easily show the global existence of the unique solution $z$ to \eqref{ODE}--\eqref{ODEinit}, and we obtain \eqref{ODEBound}.

Now we turn our attention to the asymptotic behavior of $z$. 
In what follows, $C$ stands for various positive constants
that can be determined only by 
$K_0$, $E_0$, $c_0$, $\rho$, and $\mu$.
The actual value of $C$ may change line by line.

Let $\xi$ and $\eta$
be the solution to the system of ODEs
\begin{align}
& \xi'(t)= -i \frac{\imagpart K}{2t\eta(t)} |\xi(t)|^2 \xi(t)
  +{J}(t)\sqrt{\eta(t)}, && t>t_0,
\label{EqXi}
\\
& \eta'(t)=\frac{\realpart K}{t}|\xi(t)|^2, && t>t_0,
\label{EqEta}\\
& \xi(t_0)=z_0,\quad \eta(t_0)=1.
\label{SysData}
\end{align}
We can easily check that  $z(t)={\xi(t)}/{\sqrt{\eta(t)}}$
as long as $(\xi, \eta)$ exists.

By the local existence theorem for ODEs, there is a local solution 
$(\xi, \eta)$ on
 $[t_0, T_1)$ with some $T_1(>t_0)$. For $t_0<\tau_0<T_1$, we put
$$
M_{\tau_0}:=\sup_{t_0\le t\le \tau_0} |\xi(t)|.
$$
Then we get $0\le \eta'(t)\le K_0 M_{\tau_0}^2t^{-1}$ for 
$t_0\le t\le {\tau_0}$, which leads to 
\begin{equation}
\label{BoundEta}
1\le \eta(t) \le 1+K_0 M_{\tau_0}^2 \log(t/t_0), \quad t_0\le t\le \tau_0.
\end{equation}
By \eqref{EqXi} we obtain
$$
\frac{d}{dt}|\xi(t)|^2=
2\realpart\left(\overline{\xi(t)} \, J(t)\sqrt{\eta(t)}\right)
\le 2|\xi(t)||{J}(t)|\sqrt{\eta(t)} ,
$$
which, in combination with \eqref{BoundEta}, yields
\begin{align}
|\xi(t)|
\le & |z_0|+\int_{t_0}^t |J(\tau)| \sqrt{\eta(\tau)} d\tau
\nonumber\\
\le & C\eps\left(\jb{\sigma}^{-\nu}
+\jb{\sigma}^{-\kappa}\int_{t_0}^t \tau^{-\rho}
\left(1+M_{\tau_0} \sqrt{\log \frac{\tau}{t_0}}\right)d\tau\right)
\label{B401}
\end{align}
for $t_0\le t\le \tau_0$. Since 
we have $\log t\le C_\mu t^{2\mu}$ for $t\ge 1$ with a positive constant $C_\mu$,
and since $-\rho+\mu+1<0$,
we get
$$
\int_{t_0}^t \tau^{-\rho}\sqrt{\log \frac{\tau}{t_0}}d\tau\le Ct_0^{-\mu}\int_{t_0}^\infty \tau^{-\rho+\mu}d\tau\le Ct_0^{-\rho+1}
\le C\jb{\sigma}^{-\rho+1},\quad t\ge t_0.
$$
Hence we obtain from \eqref{B401} that
\begin{equation}
\label{B402}
M_{\tau_0}\le C\eps\jb{\sigma}^{-\nu}+C\eps \jb{\sigma}^{-\nu} M_{\tau_0}\le C\eps\jb{\sigma}^{-\nu}+\frac{1}{2}M_{\tau_0}
\end{equation}
for $0<\eps\le \eps_2:=1/(2C)$ (with the constant $C$ from \eqref{B402}),
which yields
\begin{equation}
\label{BoundXi}
\left(\sup_{t_0\le t\le \tau_0}|\xi(t)|=\right)M_{\tau_0} \le C\eps\jb{\sigma}^{-\nu},
\end{equation}
provided that $0<\eps\le \eps_2$.
With the {\it a priori} bound \eqref{BoundXi} as well as \eqref{BoundEta},
we see that the solution $(\xi, \eta)$ exists globally in time, and
we also have
\begin{equation}
\label{BoundXiEta}
|\xi(t)|\le C\eps\jb{\sigma}^{-\nu},\ 1\le\eta(t)\le 1+C \eps^2\log \frac{t}{t_0}
\le 1+C\eps^2 \left(\frac{t}{t_0}\right)^{2\mu} 
\end{equation}
for $t\ge t_0$, provided that $0<\eps\le \eps_2$.

We assume $0<\eps\le \eps_2$ from now on.
We put
$$
\Theta(t)=\frac{\imagpart K}{2}\int_{t_0}^t \frac{|\xi(\tau)|^2}{\tau\eta(\tau)}d\tau,
\quad t\ge t_0.
$$
Then \eqref{EqXi} implies
$$
\xi(t)=e^{-i\Theta(t)}\left(z_0+
\int_{t_0}^t e^{i\Theta(\tau)}J(\tau)\sqrt{\eta(\tau)}d\tau
\right),
\quad t\ge t_0.
$$
We define
$$
z_+:= z_0+\int_{t_0}^\infty 
e^{i\Theta(\tau)} J(\tau)\sqrt{\eta(\tau)}d\tau,
\quad \xi_+(t):=e^{-i\Theta(t)}z_+.
$$
Then we obtain
\begin{align}
|\xi_+(t)-\xi(t)|\le & \int_{t}^\infty| J(\tau)|
\sqrt{\eta(\tau)}d\tau\nonumber\\
\le & C\eps\jb{\sigma}^{-\kappa}\int_t^\infty \tau^{-\rho}
\left(1+\eps_2 \left(\frac{\tau}{t_0}\right)^\mu\right)d\tau\nonumber\\
\le & C \eps (\jb{\sigma}^{-\kappa}t^{-\rho+1}+\jb{\sigma}^{-\kappa-\mu} t^{-\rho+\mu+1})\nonumber\\
\le & C\eps \jb{\sigma}^{-\kappa-\mu}t^{-\rho+\mu+1}, \quad t\ge t_0.
\label{B454}
\end{align}
Especially we have
$$
|z_+-z_0|\le C\eps \jb{\sigma}^{-\kappa-\mu}t_0^{-\rho+\mu+1}
\le C\eps \jb{\sigma}^{-\nu},
$$
which shows
\begin{equation}
\label{B456}
|z_+|\le C\eps\jb{\sigma}^{-\nu}.
\end{equation}

Observing that $|z_+|=|\xi_+(t)|$, we obtain from \eqref{B454} and \eqref{B456} that
\begin{align}
\label{B457}
\left||\xi(t)|^2-|z_+|^2\right|\le & (|\xi(t)|+|z_+|)|\xi(t)-\xi_+(t)|
\le C\eps^2\jb{\sigma}^{-\nu-\kappa-\mu}t^{-\rho+\mu+1}
\end{align}
for $t\ge t_0$.
We define
$$
\eta_\infty(t):=1+(\realpart K)\left(|z_+|^2 \log \frac{t}{t_0}
+\int_{t_0}^\infty \frac{|\xi(\tau)|^2-|z_+|^2}{\tau}d\tau\right), \quad t\ge 1.
$$
Since \eqref{EqEta} implies that
$$
\eta(t)=1+(\realpart K)\left(|z_+|^2 \log \frac{t}{t_0}
+\int_{t_0}^t \frac{|\xi(\tau)|^2-|z_+|^2}{\tau}d\tau\right), \quad t\ge t_0
$$
we obtain from \eqref{B457} that
\begin{align}
\label{B458}
|\eta(t)-\eta_\infty(t)|\le K_0\int_t^\infty\frac{\bigl||\xi(\tau)|^2-|z_+|^2\bigr|}{\tau}d\tau\le C\eps^2 \jb{\sigma}^{-\nu-\kappa-\mu}t^{-\rho+\mu+1}
\end{align}
for $t\ge t_0$.
Especially we have $\eta_\infty(t_0)\ge 1-C\eps^2$, which leads to
\begin{equation}
\label{B459}
\eta_\infty(t)\ge 1-C\eps^2-K_0|z_+|^2\log t_0\ge 1-C\eps^2-C \eps^2 \jb{\sigma}^{-2\nu}\log t_0, 
\quad t\ge 1
\end{equation}
with the help of \eqref{B456}.
By \eqref{CondT0},
there is a positive constant $c_1=c_1(c_0, \rho)$ such that we have
$$
\jb{\sigma}^{-2\nu}\log t_0\le \jb{\sigma}^{-2(\rho-1)} \log (c_0\jb{\sigma})\le c_1,
\quad \sigma\in \R.
$$
If we put
$$
\eps_1:=\min\left\{\sqrt{\frac{1}{4C}}, 
\sqrt{\frac{1}{4c_1C}},
\eps_2\right\}
$$
with the constant $C$ coming from \eqref{B459},
then we get 
\begin{equation}
\label{LBXi1}
\eta_\infty(t)\ge \frac{1}{2},\quad t\ge 1
\end{equation}
for $0<\eps\le \eps_1$.

From now on, we assume that $0<\eps \le \eps_1$.
We set
\begin{align*}
\Theta_\infty(t):=& \frac{\imagpart K}{2}
\int_{t_0}^t \frac{|z_+|^2}{\tau \eta_\infty(\tau)}d\tau, \quad t\ge 1\\
\Theta_0:=& \frac{\imagpart K}{2}
\int_{t_0}^\infty
\left(\frac{|\xi(\tau)|^2}{\tau\eta(\tau)}-\frac{|z_+|^2}{\tau\eta_\infty(\tau)}\right)d\tau.
\end{align*}
From \eqref{BoundXiEta}, \eqref{B456}, \eqref{B457}, \eqref{B458}, and \eqref{LBXi1}, we obtain
\begin{align*}
\left|\frac{|\xi(\tau)|^2}{\eta(\tau)}-\frac{|z_+|^2}{\eta_\infty(\tau)}\right|
\le & \left|\frac{|\xi(\tau)|^2-|z_+|^2}{\eta(\tau)}\right|+|z_+|^2\left|
\frac{1}{\eta(\tau)}-\frac{1}{\eta_\infty(\tau)}\right|
\\
\le & C\eps^2\jb{\sigma}^{-\nu-\kappa-\mu}\tau^{-\rho+\mu+1}
\end{align*}
for $\tau\ge t_0$, which yields
\begin{align}
|(\Theta_\infty(t)+\Theta_0)-\Theta(t)|\le & \frac{K_0}{2}
\int_{t}^\infty
\left|\frac{|\xi(\tau)|^2}{\tau\eta(\tau)}-\frac{|z_+|^2}{\tau\eta_\infty(\tau)}\right|d\tau
\nonumber\\
\le & C\eps^2 \jb{\sigma}^{-\nu-\kappa-\mu} t^{-\rho+\mu+1}, \quad t\ge t_0.
\label{Ma}
\end{align}
We define
$$
\xi_\infty(t):=e^{-i\Theta_\infty(t)}(e^{-i\Theta_0}z_+),\quad t\ge 1.
$$
Then \eqref{Ma} leads to
\begin{align}
|\xi_\infty(t)-\xi_+(t)|\le & |e^{-i(\Theta_\infty(t)+\Theta_0)}-e^{-i\Theta(t)}|\,|z_+|
\nonumber\\
\le & |(\Theta_\infty(t)+\Theta_0)-\Theta(t)|\,|z_+|
\le 
C\eps^3 \jb{\sigma}^{-2\nu-\kappa-\mu} t^{-\rho+\mu+1}, 
\end{align}
which, together with \eqref{B454}, yields
\begin{equation}
\label{B480}
|\xi_\infty(t)-\xi(t)|\le  
C\eps \jb{\sigma}^{-\kappa-\mu}t^{-\rho+\mu+1},\quad t\ge t_0.
\end{equation}

Observing that $|\xi_\infty(t)|=|z_+|$ by definition, we find that
\begin{align*}
\xi_\infty'(t)=& -i\Theta_\infty'(t)\xi_\infty(t)= -i\frac{\imagpart K}{2t\eta_\infty(t)} |\xi_\infty(t)|^2\xi_\infty(t),  \\
\eta_\infty'(t)=& \frac{\realpart K}{t}|\xi_\infty(t)|^2 
\end{align*}
for $t>1$. Hence if we put $z_\infty(t)=\xi_\infty(t)/\sqrt{\eta_\infty(t)}$,
then we get
$$
z_\infty'(t)= -\frac{K}{2t}|z_\infty(t)|^2z_\infty(t),\quad t>1.
$$
It follows from \eqref{BoundXiEta}, \eqref{B456}, \eqref{B458}, \eqref{LBXi1}, and \eqref{B480} that
\begin{align}
|z(t)-z_\infty(t)|
\le & \frac{|\xi(t)-\xi_\infty(t)|}{\sqrt{\eta(t)}}+|z_+|
\left|\frac{1}{\sqrt{\eta(t)}}-\frac{1}{\sqrt{\eta_\infty(t)}}\right| 
\nonumber\\
\le & 
C\eps \jb{\sigma}^{-\kappa-\mu} t^{-\rho+\mu+1}, \quad t\ge t_0.
\label{B481}
\end{align}

Finally we put $p_0:=z_\infty(1)$, and let $p$ be the solution to
\eqref{ReducedODE}.
It is clear that $z_\infty(t)=p(\log t)$ for $t\ge 1$.
By \eqref{B456} and \eqref{LBXi1}, we get
\begin{equation}
\label{B482}
|p_0|=\left|\frac{\xi_\infty(1)}{\sqrt{\eta_\infty(1)}}\right|\le C\eps\jb{\sigma}^{-\nu}.
\end{equation}
We obtain \eqref{ODEaSymp} and \eqref{ODEprof} from \eqref{B481} and \eqref{B482}.
This completes the proof.
\end{proof}
\section{Proof of Theorem~\ref{Global}}
\label{Proof0101}
In this section we prove Theorem~\ref{Global}.
Let $u$ be a smooth solution to \eqref{WaveEq}--\eqref{InitData} on 
$[0,T)\times \R^2$ for some $T>0$.
For a positive integer $k$, and positive constants $\lambda$ and $\mu$, we define
\begin{align*}
 e_{k,\lambda,\mu}[u](T):=& \sup_{(t,x)\in [0,T)\times \R^2}\bigl\{
(1+t)^{1/2} \jb{t-r}^{1-\mu}|\pa u(t,x)|\\
& \qquad\qquad\qquad  {}+(1+t)^{(1-\lambda)/2}\jb{t-r}^{1-\mu}|\pa u(t,x)|_k\bigr\},
\end{align*}
where $r=|x|$. 
Our first aim here is to prove the following,
from which the global existence part of Theorem~\ref{Global} follows:
\begin{prp}\label{AP01}
Suppose that the assumptions in Theorem~$\ref{Global}$ are fulfilled.
Let $k\ge 2$ and $0<\mu<1/10$.
If $0< 2k\lambda \le \mu/6$, then 
we can find a positive constant $m_0=m_0(k, \lambda, \mu)$ having the following
property:
For any $m\ge m_0$ 
there is a positive constant $\eps_0=\eps_0(m, k, \lambda, \mu)$
such that 
\begin{equation}
\label{Boot01}
e_{k,\lambda,\mu}[u](T)\le m\eps
\end{equation}
implies
\begin{equation}
\label{Boot-F}
e_{k,\lambda,\mu}[u](T)\le \frac{m}{2}\eps,
\end{equation}
provided that $0<\eps\le \eps_0$.
\end{prp}
%
{\it Proof.}
Assume that \eqref{Boot01} is satisfied.
In the following we always suppose that $0\le t<T$.
We also suppose that $m\ge 1$, and that $\eps$ is small enough to satisfy $m\eps \le 1$.
The letter $C$ in this proof stands for a positive constant which may depend
on $k$, $\lambda$, and $\mu$, but are independent of $m$, $\eps$, and $T$.
The proof is divided into several steps.
\medskip

{\bf Step 1: The energy estimates.}
For $l\le 2k$, it follows from \eqref{Boot01} that
\begin{align}
|F(\pa u)|_{l}\le & C\left(|\pa u|^2|\pa u|_l+|\pa u|_k^2|\pa u|_{l-1}\right)
\nonumber\\
\le & \frac{C^*}{2}m^2\eps^2(1+t)^{-1}|\pa u|_l+Cm^2\eps^2(1+t)^{\lambda-1}|\pa u|_{l-1}
\label{Boot02}
\end{align}
with a positive constant $C^*=C^*(k,\lambda,\mu)$,
where terms including $|\pa u|_{l-1}$ should be neglected if $l=0$. 

From the energy inequality and \eqref{Boot02} with $l=0$, we get
$$
\|\pa u(t)\|_{0}\le C\eps+C^* m^2\eps^2\int_0^t (1+\tau)^{-1}\|\pa u(\tau)\|_{0} d\tau.
$$
Gronwall's lemma implies
\begin{equation}
\|\pa u(t)\|_{0}\le C\eps (1+t)^{C^*m^2\eps^2}.
\label{Ind00}
\end{equation}
We are going to prove that there are positive constants $B_l=B_l(k, \lambda, \mu)$
for $0\le l\le 2k$ such that
\begin{equation}
\label{Ind}
\|\pa u(t)\|_l \le B_l \eps (1+t)^{C^*m^2\eps^2+\lambda l}
\end{equation}
for $0\le l\le 2k$.
Indeed \eqref{Ind} for $l=0$ follows from \eqref{Ind00}.
Suppose that \eqref{Ind} is true for some $0\le l\le 2k-1$. 
Then, by \eqref{Boot02} we get
\begin{align*}
\|\pa u(t)\|_{l+1}\le & C\eps+C^*m^2\eps^2\int_0^t  (1+\tau)^{-1}\|\pa u(\tau)\|_{l+1} d\tau\\
& {}+CB_l m^2\eps^3\int_0^t (1+\tau)^{C^*m^2\eps^2+\lambda(l+1)-1}d\tau.
\end{align*}
Then Gronwall's lemma yields
\begin{align*}
\|\pa u(t)\|_{l+1}\le C\eps (1+t)^{C^*m^2\eps^2}+C\frac{B_l}{\lambda(l+1)}m^2\eps^3(1+t)^{C^*m^2\eps^2+\lambda(l+1)},
\end{align*}
which inductively implies the desired result because $m\eps\le 1$.
\medskip

{\bf Step 2: Decay estimates of generalized derivatives of higher order.}
We suppose that $\eps$ is so small to satisfy 
$C^*m^2\eps^2 \le \mu/6$, where the constant $C^*$ is from \eqref{Ind}. 
Then \eqref{Ind} for $l=2k$ implies
\begin{equation}
\|\pa u(t)\|_{2k}\le C\eps (1+t)^{\mu/3},
\label{Hoe00}
\end{equation}
because we have assumed $2k\lambda\le \mu/6$.
\eqref{Hoe00} yields
\begin{align*}
\int_{\R^2}
|F(\pa u)(\tau,y)|_{2k}dy \le & Cm\eps (1+\tau)^{(\lambda-1)/2}\|\pa u(\tau)\|_{2k}^2
\\
\le & Cm\eps^3(1+\tau)^{\mu-(1/2)} \le C\eps (1+\tau)^{\mu-(1/2)}, 
\end{align*}
because $\lambda/2<\mu/3$.
It is well known that we have 
$|u_0(t,x)|_{2k-1}\le C\eps(1+t)^{-1/2}$
for the solution $u_0$ to $\dal u_0=0$ with initial data $(u_0,\pa_t u_0)
=(\eps f, \eps g)$ at $t=0$ (see \cite{Hoe97} for instance).
Hence, by Lemma~\ref{Ho} we get
\begin{align}
(1+t)^{1/2}|u(t,x)|_{2k-1}\le & C\eps+C\int_0^t \int_{\R^2} \frac{|F(\pa u)(\tau, y)|_{2k}}{(1+\tau)^{1/2}}dy d\tau\nonumber\\
\le & C\eps(1+t)^{\mu}
\label{Hoe01}
\end{align}
for $(t,x)\in [0,T)\times \R^2$.
From Lemma~\ref{Rewrite03} we obtain
\begin{equation}
\label{Kla01}
|\pa u(t,x)|_{2k-2}\le C\jb{t-r}^{-1}|u(t,x)|_{2k-1}\le C\eps (1+t)^{\mu-(1/2)}\jb{t-r}^{-1}
\end{equation}
for $(t, x)\in [0,T)\times \R^2$.


Suppose that $R$ is a positive number to satisfy \eqref{Supp01},
so that we have \eqref{Supp02}. Recall the definition \eqref{DefLamTR} of $\Lambda_{T,R}$.
We put $\Lambda_{T,R}^{\rm c}:=\bigl([0,T)\times \R^2\bigr)\setminus \Lambda_{T,R}$.

If we have either $t/2<1$ or $r< t/2$, then we get 
$$
\jb{t-r}^{-1}\le C\jb{t+r}^{-1}.
$$
If $r>t+R$, then we have $u(t,x)=0$ by \eqref{Supp02}.
Hence \eqref{Kla01} implies
\begin{equation}
\label{BF01}
\sup_{(t,x)\in \Lambda_{T,R}^c}(1+t)^{1/2}\jb{t-r}^{1-\mu}|\pa u(t,x)|_{2k-2}\le C\eps.
\end{equation}
\medskip

{\bf Step 3: Decay estimates for generalized derivatives of lower order.}
We suppose that $(t,x)=(t,r\omega)\in \Lambda_{T,R}$ throughout this step.
Recall that $t^{-1}$, $r^{-1}$, $(1+t)^{-1}$, and $\jb{t+r}^{-1}$ are equivalent to each other.
We define $U$, $U^{(\alpha)}$, $H$, and $H_\alpha$ as in Section~\ref{Reduction}.

By \eqref{Hoe01} and \eqref{Kla01}, we get
\begin{equation}
\label{E501}
|u(t,x)|_{\sharp,2k-2}\le C\eps t^{\mu-(1/2)}\jb{t-r}^{-1},
\end{equation}
where $|u(t,x)|_{\sharp, 2k-2}$ is defined by \eqref{Basis300}.
Hence we obtain from \eqref{Basic301} that 
$$
\sum_{|\alpha|\le 2k-2}|U^{(\alpha)}(t,x)|\le C\eps t^{\mu}\jb{t-r}^{-1}.
$$
If $t/2=r$ or $t/2=1$, then we have 
$t^{\mu}\le C\jb{t-r}^{\mu}$,
and we obtain 
\begin{equation}
\label{BF03}
\sum_{|\alpha|\le 2k-2}|U^{(\alpha)}(t,x)|\le C\eps \jb{t-r}^{\mu-1}
\text{ when either $t/2=r$ or $t/2=1$}.
\end{equation}

By Corollary~\ref{Rewrite02} and \eqref{Hoe01}, we get
\begin{align}
\label{f401}
t^{1/2}|\pa u|_l\le & C \sum_{|\alpha|\le l} |r^{1/2}\pa \Gamma^\alpha u| 
\le C\sum_{|\alpha|\le l} |U^{(\alpha)}|+C
\eps t^{\mu-1},\quad l\le 2k-2.
\end{align}

Recall that $0<\mu<1/10$.
By \eqref{Hoe01}, \eqref{E501}, and Proposition~\ref{Pro301} we get
\begin{equation}
|H(t,x)|\le C\left(\eps^3 t^{3\mu-2}\jb{t-r}^{-2}
{}+\eps t^{\mu-2}\right)
\le C\eps t^{3\mu-2}\jb{t-r}^{-2\mu}.
 \label{E502}
\end{equation}

We define $t_{0, \sigma}=\max\{2, -2\sigma\}$ for $\sigma\le R$.
Note that the line segment $\{(t, (t+\sigma)\omega);\, 0\le t< T\}$, with $\sigma\le R$ and $\omega\in \Sph^1$ being fixed, meets the boundary of $\Lambda_{T,R}$
at the point $\left(t_{0, \sigma}, (t_{0,\sigma}+\sigma)\omega\right)$. 
We have
\begin{equation}
\label{517}
(1+R)^{-1}\jb{\sigma} \le t_{0,\sigma}\le 2(1+|\sigma|)\le 2\sqrt{2} \jb{\sigma},\quad \sigma\le R.
\end{equation}

We define
$$
V_{\sigma,\omega}(t)=U\bigl(t, (t+\sigma)\omega\bigr)
$$
for $0\le t\le T$, $\sigma\le R$, and $\omega\in \Sph^1$.
Then \eqref{Red01} leads to
\begin{equation}
\label{Red01a}
V_{\sigma, \omega}'(t)= -\frac{1}{2t}F(\hat{\omega})|V_{\sigma,\omega}(t)|^2V_{\sigma, \omega}(t)+H(t, (t+\sigma)\omega)
\end{equation}
for $t_{0,\sigma}\le t<T$.
Note that by \eqref{BF03} we have
\begin{equation}
\label{BF04}
|V_{\sigma, \omega}(t_{0,\sigma})|
\le C\eps \jb{\sigma}^{\mu-1}.
\end{equation}
for $\sigma\le R$ and $\omega\in \Sph^1$.
By \eqref{E502}, \eqref{517}, and \eqref{BF04}, we can apply 
Lemma~\ref{ODELemma02} to \eqref{Red01a} (with $\rho=2-3\mu$, and $\kappa=2\mu$):
\eqref{ODEBound}  implies that
\begin{equation}
\label{E409}
|V_{\sigma, \omega}(t)|\le C\eps \jb{\sigma}^{\mu-1},\quad t\ge t_{0,\sigma},
\end{equation}
where $C$ is a constant independent of $\sigma$ and $\omega$
(note that $0<\mu<1/10<\rho-1$ and $\kappa+\rho-1=1-\mu$).
Now we get $|U(t,r\omega)|=|V_{r-t,\omega}(t)|\le C\eps\jb{t-r}^{\mu-1}$,
and with the help of \eqref{f401} for $l=0$ we obtain
\begin{equation}
\label{f410}
\sup_{(t,x)\in \Lambda_{T,R}} (1+t)^{1/2}\jb{t-r}^{1-\mu} |u(t,x)|
\le C\eps.
\end{equation}

Let $|\alpha|\le k$.
For a nonnegative integer $s$, we set 
$$
{\mathcal U}^{(s)}(t,x):=\sum_{|\alpha|\le s}|U^{(\alpha)}(t,x)|.
$$
By \eqref{f401} we get
\begin{equation}
 |\pa u(t,x)|_{|\alpha|-1}\le C\left(t^{-1/2}{\mathcal U}^{(|\alpha|-1)}(t,x)+\eps t^{\mu-(3/2)}\right).
 \label{B410}
\end{equation}
We obtain from \eqref{E501}, \eqref{B410}, and Proposition~\ref{Pro301} that
\begin{align}
|H_\alpha(t,x)|\le & C\left(
 \eps^3 t^{3\mu-2}\jb{t-r}^{-2}
{}+\eps t^{\mu-2}+\eps^3t^{3\mu-4}+t^{-1}
\bigl({\mathcal U}^{(|\alpha|-1)}(t,x)\bigr)^3\right),
\nonumber\\
\le & C\eps t^{3\mu-2}\jb{t-r}^{-2\mu}+Ct^{-1}\bigl({\mathcal U}^{(|\alpha|-1)}(t,x)\bigr)^3.
\label{f421}
\end{align} 

We put
$$
V^{(\alpha)}_{\sigma, \omega}(t)=U^{(\alpha)}\bigl(t, (t+\sigma)\omega\bigr)
$$
for $0\le t< T$, $\sigma\le R$, and $\omega\in \Sph^1$.
From \eqref{Red02} we get
\begin{align*}
\bigl(V^{(\alpha)}_{\sigma,\omega}\bigr)'(t)=& -\frac{F(\hat{\omega})}{2t}
\left(2|V_{\sigma,\omega}(t)|^2 V^{(\alpha)}_{\sigma,\omega}(t)
+\bigl(V_{\sigma,\omega}(t)\bigr)^2\overline{V^{(\alpha)}_{\sigma,\omega}(t)}\right)\\
&+H_\alpha(t, (t+\sigma) \omega)
\end{align*}
for $t_{0,\sigma}\le t<T$.
Hence by \eqref{BF04} and \eqref{f421} we obtain
\begin{align*}
\frac{d}{dt}|V^{(\alpha)}_{\sigma,\omega}(t)|^2=& 
 -\realpart\left(\frac{F(\hat{\omega})}{t}
\left(2|V_{\sigma,\omega}(t)|^2 |V^{(\alpha)}_{\sigma,\omega}(t)|^2
+\bigl(V_{\sigma,\omega}(t)\bigr)^2\bigl(\overline{V^{(\alpha)}_{\sigma,\omega}(t)}\bigr)^2\right)
\right)\\
&+2\realpart\left(H_\alpha(t, (t+\sigma)\omega)
\overline{V^{(\alpha)}_{\sigma,\omega}}(t)\right)\\
\le & \frac{2C_*\eps^2}{t} |V^{(\alpha)}_{\sigma,\omega}(t)|^2
{}+C\left(\eps t^{3\mu-2}\jb{\sigma}^{-2\mu}+t^{-1}\bigl({\mathcal V}_{\sigma, \omega}^{(|\alpha|-1)}(t)\bigr)^3 \right)\,|V^{(\alpha)}_{\sigma,\omega}(t)|
\end{align*}
for $t_{0,\sigma}\le t<T$, where ${\mathcal V}_{\sigma,\omega}^{(s)}(t):=\sum_{|\alpha|\le s}|V_{\sigma, \omega}^{(\alpha)}(t)|$, and $C_*=C_*(k,\lambda, \mu)$ is a positive constant.
Therefore it follows from \eqref{BF03} that
\begin{align*}
t^{-C_*\eps^2}|V^{(\alpha)}_{\sigma,\omega}(t)|
\le & t_{0,\sigma}^{-C_*\eps^2}
|V^{(\alpha)}_{\sigma,\omega}(t_{0,\sigma})|
{}+C\eps \jb{\sigma}^{-2\mu} \int_{t_{0,\sigma}}^t \tau^{-C_*\eps^2+3\mu-2}d\tau
\\
&{}+C\int_{t_{0,\sigma}}^t \tau^{-C_*\eps^2-1}
\bigl({\mathcal V}_{\sigma, \omega}^{(|\alpha|-1)}(\tau)\bigr)^3 d\tau\\
\le & C\eps\jb{\sigma}^{\mu-1-C_*\eps^2}+C\int_{t_{0,\sigma}}^t \tau^{-C_*\eps^2-1}
\bigl({\mathcal V}_{\sigma, \omega}^{(|\alpha|-1)}(\tau)\bigr)^3 d\tau
\end{align*}
for $t_{0,\sigma}\le t <T$,
which leads to
$$
t^{-C_*\eps^2}{\mathcal V}^{(l)}_{\sigma,\omega}(t)
\le C\eps\jb{\sigma}^{\mu-1}+C\int_{t_{0,\sigma}}^t \tau^{-C_*\eps^2-1}
\bigl({\mathcal V}_{\sigma, \omega}^{(l-1)}(\tau)\bigr)^3 d\tau
$$
for $t_{0,\sigma}\le t<T$ and $1\le l \le k$.
Using this inequality, we are going to prove that
\begin{align}
\label{BF}
{\mathcal V}_{\sigma, \omega}^{(l)}(t)
\le \widetilde{B}_l \eps t^{3^{l-1} C_* \eps^2} \jb{\sigma}^{\mu-1} 
\end{align}
for $t_{0,\sigma}\le t < T$ and $1\le l\le k$
with some positive constant $\widetilde{B}_l=\widetilde{B}_l(k,\lambda, \mu)$.
By \eqref{E409} 
we have ${\mathcal V}_{\sigma, \omega}^{(0)}(t)\le \widetilde{B}_0\eps\jb{\sigma}^{\mu-1}$
with a positive constant $\widetilde{B}_0=\widetilde{B}_0(k,\lambda,\mu)$.
Hence we get
\begin{align*}
t^{-C_*\eps^2}{\mathcal V}^{(1)}_{\sigma,\omega}(t)
\le & C\eps\jb{\sigma}^{\mu-1}+C\widetilde{B}_0^3\eps^3\jb{\sigma}^{3\mu-3}\int_{2}^\infty \tau^{-C_*\eps^2-1} d\tau\le \widetilde{B}_1\eps\jb{\sigma}^{\mu-1}
\end{align*}
with $\widetilde{B}_1=C+CC_*^{-1}\widetilde{B}_0^3$,
which leads to 
$$
{\mathcal V}_{\sigma, \omega}^{(1)}(t)\le \widetilde{B}_1 \eps \jb{\sigma}^{\mu-1} t^{C_*\eps^2},
$$
and \eqref{BF} for $l=1$ is shown.
Next, suppose that \eqref{BF} is true for some $l$ with $1\le l \le k-1$.
Then we get
\begin{align*}
t^{-C_*\eps^2}{\mathcal V}^{(l+1)}_{\sigma,\omega}(t)
\le & C\eps\jb{\sigma}^{\mu-1}+C\widetilde{B}_l^3 \eps^3\jb{\sigma}^{3\mu-3}
\int_{2}^t \tau^{(3^{l}-1)C_*\eps^2-1} d\tau\\
\le & \widetilde{B}_{l+1}\eps \jb{\sigma}^{\mu-1} t^{(3^{l}-1)C_*\eps^2}
\end{align*}
with $\widetilde{B}_{l+1}=C+CC_*^{-1} (3^{l}-1)^{-1}\widetilde{B}_l^{3}$,
and we obtain
$$
{\mathcal V}_{\sigma, \omega}^{(l+1)}(t) \le \widetilde{B}_{l+1}\eps \jb{\sigma}^{\mu-1}
t^{3^{l}C_*\eps^2},
$$
which is \eqref{BF} with $l$ replaced by $l+1$.
Now \eqref{BF} for $1\le l \le k$ is established.

By \eqref{BF}, we obtain
$$
\sum_{|\alpha|\le k}|U^{(\alpha)}(t,x)|\le C\eps (1+t)^{3^{k-1}C_*\eps^2}\jb{t-r}^{\mu-1},\quad (t,x)\in \Lambda_{T,R},
$$
which yields 
\begin{equation}
\label{BF404}
|\pa u(t,x)|_k\le C\eps (1+t)^{3^{k-1}C_*\eps^2-(1/2)}\jb{t-r}^{\mu-1},\quad (t,x)\in \Lambda_{T,R}
\end{equation}
with the help of \eqref{f401}.
If we choose sufficiently small $\eps$ to satisfy
$$
3^{k-1}C_*\eps^2\le \lambda/2,
$$
it follows from \eqref{BF404} that
\begin{equation}
\label{BFF}
\sup_{(t,x)\in \Lambda_{T, R}} (1+t)^{(1-\lambda)/2}\jb{t-r}^{1-\mu}|\pa u(t,x)|_k\le C\eps.
\end{equation}
\medskip

{\bf The final step.}
By \eqref{BF01}, \eqref{f410}, and \eqref{BFF}, 
we see that there exist two positive constants $\eps_0=\eps_0(m, k, \lambda, \mu)$
and $C_0=C_0(k, \lambda, \mu)$
such that
\begin{equation}
\label{QQ}
e_{k,\lambda, \mu}[u](T)\le C_0\eps
\end{equation}
for $0<\eps\le \eps_0$ (note that we have $2k-2\ge k$ for $k\ge 2$).
If $m\ge \max\{2C_0, 1\}$, then \eqref{QQ} implies \eqref{Boot-F} immediately.
\qed
\medskip

We are in a position to prove Theorem~\ref{Global}.
\begin{proof}[Proof of Theorem~$\ref{Global}$]
Let the assumptions in Theorem~\ref{Global} be fulfilled.
Suppose that $u$ is a local solution
to \eqref{WaveEq}--\eqref{InitData} on $[0,T)\times \R^2$ for some $T>0$.
We fix $\lambda$ and $k$ as in Proposition~\ref{AP01}.
We also fix $m$ satisfying $m\ge m_0$ and
$$
\sup_{x\in \R^2}\jb{r}^{1-\mu}(|\pa u(t,x)|+|\pa u(t,x)|_k)\bigr|_{t=0}
\le \frac{m}{2}{\eps}, \quad \eps>0,
$$
where $m_0=m_0(k, \lambda, \mu)$ is from Proposition~\ref{AP01}.
Let $\eps_0=\eps_0(m, k, \lambda, \mu)$ also be from Proposition~\ref{AP01}.
We put
$$
T_*=\sup\left\{t\in [0,T)\,; e_{k, \lambda, \mu}(t)\le m\eps  \right\}.
$$
By the choice of $m$, we have $T_*>0$. Moreover we get $T_*=T$
for $0<\eps\le \eps_0$,
because if $T_*<T$ then Proposition~\ref{AP01} implies $e_{k,\lambda,\mu}(T_*)<m\eps/2$
for $0<\eps\le \eps_0$,
and we obtain from the continuity of $e_{k,\lambda,\mu}$ that 
$e_{k,\lambda,\mu}(t_*)\le m\eps$ for some $t_*>T_*$, 
which contradicts the definition of $T_*$. 
Therefore we see that $e_{k,\lambda, \mu}(t)$ cannot exceed $m\eps$ as long as the solution $u$ exists.
This {\it a priori} estimate and the local existence theorem implies the
global existence of the solution $u$. We also see that \eqref{Boot01} holds 
for some large $m$.

Now we turn our attention to the asymptotic behavior.
Because we have \eqref{Boot01},
the estimates in the proof of Proposition~\ref{AP01} are valid.
We go back to \eqref{Red01a}, and apply Lemma~\ref{ODELemma02} 
with $z(t)=V_{\sigma,\omega}(t)$, $K=F(\hat\omega)$, 
$J(t)=H(t,(t+\sigma)\omega)$, $t_0=t_{0,\sigma}$
for each fixed $(\sigma, \omega)\in (-\infty, R]\times \Sph^1$
(note that we can take $\rho=2-3\mu$ and $\kappa=2\mu$
because of \eqref{E502}, \eqref{517}, and \eqref{BF04}).  
Then we see that there exists $P_0=P_0(\sigma, \omega)$ such that
\begin{align}
 \label{B601}
& |V_{\sigma, \omega}(t)-P(\log t, \sigma, \omega)|\le C\eps t^{4\mu-1}
\jb{\sigma}^{-3\mu}, \quad t\ge t_{0,\sigma},\\
  \label{B602}
& |P_0(\sigma, \omega)|\le C\eps \jb{\sigma}^{\mu-1},
\end{align}
where $P=P(\tau, \sigma, \omega)$ is the solution to \eqref{ode},
and $C$ is a constant independent of $\sigma$ and $\omega$.
Recalling that $D_-\bigl(r^{1/2} u(t,x)\bigr)=U(t,x)=V_{r-t,\omega}(t)$
with $r=|x|$ and $\omega=|x|^{-1}x$,
we obtain from \eqref{B601} that
\begin{equation}
\label{B603}
|D_-\bigl(r^{1/2}u(t,x)\bigr)
{}-P(\log t, r-t, \omega)|\le C\eps  t^{4\mu-1}\jb{t-r}^{-3\mu}
\end{equation}
for $(t,x)\in \Lambda_{T,R}$.
By Corollary~\ref{Rewrite02} and \eqref{Hoe01} we get
\begin{equation}
\label{B604}
\bigl|r^{1/2} \pa u(t,x)-\hat{\omega}(x) D_-\bigl(r^{1/2}u(t,x)\bigr)\bigr|\le C\eps (1+t)^{\mu-1},
\end{equation}
where $\hat{\omega}(x)=(-1, x/|x|)$.
Since \eqref{B602} and \eqref{ExpS1} yield 
$|P(\tau,\sigma, \omega)|\le C\eps \jb{\sigma}^{\mu-1}$,
we have
\begin{align}
|(r^{-1/2}-t^{-1/2})P(\log t, r-t, \omega)|\le & C\eps \frac{|t-r|}{\sqrt{t}\sqrt{r}(\sqrt{t}+\sqrt{r})}\jb{t-r}^{\mu-1}
\nonumber\\
\le &
C\eps t^{-3/2}\jb{t-r}^{\mu}
\label{B605}
\end{align}
for $(t,x)\in \Lambda_{T,R}$.

To sum up the estimates \eqref{B603}, \eqref{B604} and \eqref{B605}, 
we arrive at
\begin{equation}
\label{B606}
|\pa u(t,x)
{}-\hat{\omega}(x) t^{-1/2}P(\log t, r-t, \omega)|
\le C\eps t^{4\mu-(3/2)}\jb{t-r}^{-3\mu}
\end{equation}
for $(t,x)\in \Lambda_{T,R}$. 
In order to extend \eqref{B606} outside of 
$\Lambda_{T,R}$, we just have to extend the definition of $P_0$ by setting 
$P_0(\sigma, \omega)=0$ for $\sigma>R$. 
Indeed, both $u(t,x)$ and $P(\log t, r-t, \omega)$ vanish for $|x|\ge t+R$.
On the other hand, if $r<t/2$ or $1< t<2$, we have
\begin{equation}
|\pa u(t,x)|\le C\eps t^{-1/2}\jb{t-r}^{\mu-1}\le C\eps t^{\mu-(3/2)},
\end{equation}
and
\begin{align}
t^{-1/2}|P(\log t, r-t, \omega)| \le & C\eps t^{-1/2}\jb{t-r}^{\mu-1}
\le C\eps t^{\mu-(3/2)}.
\end{align}
Therefore \eqref{B606} is valid for all $(t,x)\in [1, \infty)\times \R^2$,
which shows \eqref{Concl01}.
This completes the proof.
\end{proof}
%
\section{Proof of Corollary~\ref{EnergyDecay}}\label{Omake}
Let the assumptions of Corollary~\ref{EnergyDecay} be fulfilled.
Suppose that we have \eqref{Supp01} for some $R>0$. 
Then we get \eqref{Supp02}.
We always assume $t\ge 2$ throughout this proof.
Since we have
$R+|t-r|\le \sqrt{2}(1+R)\jb{t-r}$,
\eqref{DecayDamping} leads to
\begin{equation}
\label{DecayDamping02}
|\pa u(t,x)|\le C t^{-1/2} \min\{(\log t)^{-1/2}, \eps(R+|t-r|)^{\mu-1}\},
\end{equation}
where $r=|x|$.

We put $m_\eps(t):=\eps^{1/(1-\mu)}(\log t)^{1/(2-2\mu)}$.
Here $m_\eps(t)$ is chosen so that we have
$\eps (R+t-r)^{\mu-1}=\eps m_\eps(t)^{\mu-1}=(\log t)^{-1/2}$ for
$r=t+R-m_\eps(t)$.
For small $\eps>0$ we have $0<m_\eps(t)<t$, and we get 
$0<t+R-m_\eps (t)\le t+R$.
Then it follows from \eqref{Supp02} that
$$
\|u(t)\|_E^2=\frac{1}{2}\int_{|x|\le t+R} |\pa u(t,x)|^2 dx=I_1+I_2,
$$
where
$$
I_1=\frac{1}{2}\int_{|x|\le t+R-m_\eps(t)} |\pa u(t,x)|^2 dx,\quad I_2=\frac{1}{2}\int_{t+R-m_\eps(t)\le |x|\le t+R}
|\pa u(t,x)|^2 dx.
$$
Note that we have $t^{-1}r\le t^{-1}(t+R)\le 1+R/2$ for $0\le r\le t+R$.
Hence, switching to the polar coordinate, we get from \eqref{DecayDamping02} that
\begin{align*}
I_1 \le & C\eps^2 \int_0^{t+R-m_\eps(t)} t^{-1} (R+|t-r|)^{2\mu-2} r dr\\
\le & C\eps^2 \int_0^{t+R-m_\eps(t)} (R+t-r)^{2\mu-2}dr\le C\eps^2m_\eps(t)^{2\mu-1}
=C\eps^{\frac{1}{1-\mu}} (\log t)^{-\frac{1-2\mu}{2-2\mu}}.
\end{align*}
Similarly it follows from \eqref{DecayDamping02} that
\begin{align*}
I_2 \le & C \int_{t+R-m_\eps(t)}^{t+R} t^{-1}(\log t)^{-1} r dr
\\
\le & C(\log t)^{-1} \int_{t+R-m_\eps(t)}^{t+R} dr
= 
C\eps^{\frac{1}{1-\mu}} (\log t)^{-\frac{1-2\mu}{2-2\mu}}.
\end{align*}
This completes the proof. \qed
\medskip
\section*{Acknowledgments}
The first author (S.~K.) is supported 
by Grant-in-Aid for Scientific Research (C) (No.~23540241),
JSPS. The third author (H.~S.) is supported by Grant-in-Aid for Young 
Scientists~(B) (No.~22740089), MEXT.

\end{document}